\documentclass[12pt,reqno]{amsart}
\usepackage{amssymb}
\usepackage{amsmath}
\usepackage{enumitem}
\usepackage{mathrsfs}
\usepackage[all]{xy}
\usepackage{hyperref}
\usepackage{marginnote}
\usepackage{xcolor}
\usepackage{microtype}
\usepackage{colonequals}
\usepackage{fullpage}

\RequirePackage{mathrsfs} 

\newtheorem{theorem}{Theorem}[section]
\newtheorem{lemma}[theorem]{Lemma}
\newtheorem{prop}[theorem]{Proposition}
\newtheorem{cor}[theorem]{Corollary}
\newtheorem{conj}[theorem]{Conjecture}

\newtheorem{remark}[theorem]{Remark}

\newtheorem{definition}[theorem]{Definition}

\numberwithin{equation}{section} \numberwithin{figure}{section}

\newcommand{\Qbar}{\overline{\mathbb Q}}

\newcommand\PP{\mathbb{P}}

\newcommand\ZZ{\mathbb{Z}}
\newcommand\Z{\mathbb{Z}}

\newcommand\Q{\mathbb{Q}}
\newcommand\QQ{\mathbb{Q}}

\newcommand\CC{\mathbb{C}}

\newcommand\calO{\mathcal{O}}
\newcommand{\calA}{\mathcal{A}}

\newcommand{\calT}{\mathcal{T}}
\newcommand{\calX}{\mathcal{X}}
\newcommand{\calY}{\mathcal{Y}}
\newcommand{\calZ}{\mathcal{Z}}
\newcommand{\calE}{\mathcal{E}}
\newcommand{\calF}{\mathcal{F}}
\newcommand{\disc}{\mathrm{Disc}}

\newcommand{\calW}{\mathcal{W}}
\newcommand{\calM}{\mathcal{M}}

\newcommand{\wtilde}{\widetilde}

 \DeclareMathOperator{\Spec}{Spec}

\begin{document}
 \title[Uniform height bounds for hyperbolic varieties]{Bounding heights uniformly in families of hyperbolic varieties}
 \author{Kenneth Ascher AND Ariyan Javanpeykar}

 \begin{abstract} We show that, assuming Vojta's height conjecture, the height of a rational point on an algebraically hyperbolic variety  can be bounded ``uniformly'' in families. This generalizes a result of Su-Ion Ih for curves of genus at least two to higher-dimensional varieties. As an application, we show that, assuming Vojta's height conjecture, the height of  a rational point  on a curve of general type is uniformly bounded. Finally, we prove a similar result for smooth hyperbolic surfaces with $c_1^2 > c_2$.
 \end{abstract}
  \maketitle
\section{Introduction}

The celebrated work of Caporaso, Harris, and Mazur \cite{chm}, sparked an interest in discovering implications of Lang's conjecture for rational points on varieties of general type. In fact, they show that Lang's conjecture implies a uniform bound, based solely on $k$ and the genus, of the number of $k$-points on a curve of general type defined over a number field $k$ (cf. \cite{av, has}). As Vojta's height conjecture (Conjecture  \ref{conj:Vojta}) implies the conjecture of Lang,  the aforementioned results show that Vojta's height conjecture also implies a  uniform version of Lang's conjecture. In particular, it seems reasonable to suspect that Vojta's height conjecture also has consequences for ``uniform'' height bounds.

However, one cannot expect  uniform height bounds in the naive sense.  
Indeed,  for all $P \in\mathbb P^2(\mathbb Q)$ and all $d\geq 4$, there is a smooth curve $X$ of degree $d$ in $\mathbb P^2_\QQ$ with $P \in X(\mathbb{Q})$.  Thus, for all $d\geq 4$, there is no real number $c>0$ depending only on $d$ such that for all smooth degree $d$ hypersurfaces $X \subset \mathbb P^2_\QQ$ and all $P \in X(\mathbb Q)$ the inequality $h(P) \leq c$ holds.   In particular, there is no real number $c>0$ such that for all smooth quartic hypersurfaces $X \subset \mathbb P^2_\QQ$ and all $P \in X(\mathbb Q)$ the inequality $h(P) \leq c$ holds. 

Thus, it is at first sight not  clear what is meant by ``uniform'' height bounds.
However, Su-Ion Ih has shown \cite{ih} that Vojta's height conjecture implies that the height of a rational point on a smooth proper curve of general type is bounded uniformly in families with the bound depending linearly on the height of the curve. Ih later showed in \cite{ih2} that the same is true for integral points on elliptic curves. 

The goal of this paper is to generalize Ih's results in \cite{ih} by investigating consequences of Vojta's height conjecture for families of (algebraically) hyperbolic varieties of general type. 
In this paper, a proper   scheme $X$ over a field $k$ is called  (algebraically) hyperbolic if all integral subvarieties of $X $ are  of general type; see Definition \ref{defn:hyperbolic}. 

In the statement of our main result we consider morphisms of algebraic stacks $f\colon X\to Y$ which are representable by schemes, i.e., for all schemes $S$ and all morphisms $S\to Y$, the algebraic stack $X\times_Y S$ is (representable by) a scheme. Furthermore, a substack of an algebraic stack is constructible if it is a finite union of locally closed substacks. Moreover, we will use the  relative discriminant $d_k(\mathcal T_P)$ of a point on an algebraic stack over a number field $k$; we refer the reader to Section \ref{section:discr} for a precise definition of the relative discriminant $d_k(\calT_P)$. Also, to state our theorem, we will use heights on stacks as discussed in Section \ref{section:heights_on_stacks}.

\begin{theorem}\label{thm:stackthm} Let $k$ be a number field and let $f \colon X \to Y$ be a proper surjective morphism of  proper   Deligne-Mumford stacks over $k$ which is representable by schemes.  Let $h$ be a height function on $X$ and let $h_Y$ be a height function on $Y$ associated to an ample divisor with $h_Y \geq 1$. 
Assume Vojta's height conjecture (Conjecture \ref{conj:Vojta}). Let $U\subset Y$ be a constructible substack such that, for all $t\in U$, the variety $X_t$ is smooth and hyperbolic.  Then there is a real number $c > 0$ depending only on $k$, $Y$, $X$, and $f$ such that, for all   $P$ in $X(k)$ with $f(P)$ in  $U$, the following inequality holds \[ h(P) \leq c  \cdot \big(h_Y(f(P)) + d_k(\calT_P)\big). \] 
	\end{theorem}

Note that Ih proves Theorem \ref{thm:stackthm} under the additional assumptions that the fibres are one-dimensional, and $Y$ is a scheme; see  \cite[Theorem~1.0.1]{ih}. If one assumes  that $Y$ is a scheme, then the discriminant term $d_k(\calT_P)$ can be omitted (as it equals zero). 

Ih's theorem for families of curves is slightly more general than Theorem \ref{thm:stackthm}, as he treats points of bounded degree, and not merely rational points. To keep the proofs slightly more transparent, we have restricted our attention to rational points. However, the transition from rational points to points of bounded degree can be made easily. Furthermore, the generalization of Ih's theoem to stacks is unavoidable if one desires applications to all curves simultaneously; see Theorem \ref{thm1} below, and the discussion following it.

One    cannot expect a stronger uniformity type statement for heights on (not necessarily hyperbolic) varieties of general type.   Indeed, if $k$ is a number field and $f \colon X\to Y$ is a smooth proper morphism of $k$-schemes whose geometric fibres are varieties of general type  and $t$ is a point in $Y$ such that $X_t$ contains a copy of $\mathbb P^1_{k(t)}$, then there is no real number $c>0$ such that for all $P \in X_t$, the inequality $h(P) \leq c \cdot h_Y(f(P))$ holds.    

Our proof of Theorem \ref{thm:stackthm} uses the recent \cite{avap}, which shows that Vojta's conjecture actually implies a version of the conjecture for stacks. Moreover, to prove  Theorem \ref{thm:stackthm} we  follow the strategy of Ih. Indeed, we combine an induction argument with an application of Vojta's conjecture to a desingularization of $X$ (Proposition \ref{prop:mainprop}). This line of reasoning was also used in Ih's work \cite{ih, ih2}. 

We argue that it is  more natural to work in the stacks setting, as this allows us to apply our results to moduli stacks of hyperbolic varieties, thus obtaining more concrete results. In fact, as a first corollary of Theorem \ref{thm:stackthm} we obtain the following uniformity statement for curves.

\begin{theorem}\label{thm1} Assume Conjecture \ref{conj:Vojta}.
	Let $g\geq 2$ be an integer and let $k$ be a number field. There is a real number $c$ depending only on $g$ and $k$ satisfying the following.
	For all smooth projective curves $X$ of genus $g$ over $k$, and all   $P$ in $X(k)$, the following inequality holds
	\[ 
	h(P) \leq c(g, k) \cdot \big(h(X) +d_k(\mathcal T_{X})\big).
	\]   
 \end{theorem}
 
 The discriminant term $d_k(\mathcal{T}_{X})$ can not be omitted in Theorem \ref{thm1} (and neither in Theorem \ref{thm:stackthm}). To explain this, for an integer $n\geq 1$, define $d_n := n^5+1$ and define the  genus $2$ curve $C_n$  by $d_n y^2 =x^5+1$. Note that the height of $C_n$ is equal to the height of $C_1$, as $C_{n,\Qbar}\cong C_{1,\Qbar}$ and the height is a ``geometric'' invariant. Let $P_n := (1,n) \in \mathbb{Q}^2$ and note that $P_n$ defines a $\mathbb{Q}$-rational point of $C_n$. Since $h(P_n)$ tends to infinity as $n$ gets larger, we can not omit  the discriminant term in Theorem \ref{thm1}.

It is not clear how to deduce Theorem \ref{thm1} from Ih's results, as Ih's results only apply to families of curves parametrized by schemes.
 
 Finally, we also obtain  a uniformity statement for certain hyperbolic surfaces.  
 \begin{theorem}\label{thm2} Assume Conjecture \ref{conj:Vojta}.
 Fix an even integer $a$ and a number field $k$.
	There is a real number $c$ depending only on $a$ and $k$ satisfying the following.
	For all smooth hyperbolic surfaces $S$ over $k$ with $c_1^2(S) = a> c_2(S)$ and all $P$ in $S(k)$, the following inequality holds
	\[ h(P) \leq c \cdot \big( h(S) + d_k(\mathcal T_{S}) \big). \]
 \end{theorem}
 
 We refer the reader to Section \ref{section:apps} for precise definitions of the height functions appearing in Theorems \ref{thm1} and \ref{thm2}. 
We prove Theorems \ref{thm1} and \ref{thm2} by applying Theorem \ref{thm:stackthm} to the universal family of the moduli space of curves and the moduli space of surfaces of general type, respectively. The    technical difficulty in applying Theorem \ref{thm:stackthm} is to prove the constructibility of the locus of points corresponding to hyperbolic varieties. In the setting of curves (Theorem \ref{thm1}) this is simple, whereas the case of surfaces (Theorem \ref{thm2}) requires deep results of Bogomolov and Miyaoka \cite{bogomolov, miyaoka}.  

Theorem \ref{thm:stackthm} applies to any family of varieties of general type for which the locus of hyperbolic varieties is constructible on the base.  However, as we show in Section \ref{section:apps}, verifying the constructibility of the latter locus is not straightforward.

We note that a conjecture of Lang (see \cite{langsurvey}) asserts that our  
 notion of hyperbolicity for $X$ is equivalent to being \emph{Brody hyperbolic}, i.e., that there are no non-constant holomorphic maps $f: \mathbb{C} \to X(\mathbb{C})$. In particular, as the property of being Brody hyperbolic is open in the analytic topology \cite{brody}, Lang's conjecture implies that the property of being hyperbolic is open in the analytic topology. In particular, assuming Lang's conjecture, if the locus of smooth projective hyperbolic surfaces is constructible in the moduli stack of smooth canonically polarized surfaces, then \cite[Expos\'e XII, Corollaire~2.3]{SGA1} implies that  it is (Zariski) open.

\subsection*{Acknowledgements}
	We would like to thank Dan Abramovich, Dori Bejleri, Marco Maculan, and Siddharth Mathur for useful comments and suggestions. We are most grateful to the referee for many comments and remarks which helped improve this paper.  K.A. was supported in part by funds from NSF grant DMS-1162367 and an NSF Postdoctoral Fellowship. A.J. gratefully acknowledges support from SFB/Transregio 45.

\section{Hyperbolicity} In this section the base field $k$ is a  field of arbitrary characteristic.

\begin{definition} Let $X$ be a proper   Deligne-Mumford stack of dimension $n$ over $k$. A divisor $D$ on $X$ is \emph{big} if $h^0(X, \calO_X(mD)) > c \cdot m^n$ for some $c>0$ and $m \gg1$. \end{definition}

Recall that a projective geometrically irreducible variety $X$ over $k$ is of general type if for a desingularization $\widetilde{X} \to X_{\textrm{red}}$ of the reduced scheme $X_{\textrm{red}}$, the sheaf $\omega_{\widetilde{X}}$ is big.  Note that, if $X$ is of general type and $\widetilde{X}\to X_{\textrm{red}}$ is any desingularization, then $\omega_{\widetilde{X}}$ is big.

\begin{definition}\label{defn:hyperbolic} A projective  scheme $X$ over $k$ is \emph{hyperbolic (over $k$)} if for  all its closed subschemes $Z$, any irreducible component of $Z_{\overline{k}}$ is  of general type. \end{definition}

Note that, if $X$ is a hyperbolic projective scheme over $k$, then $X$ and all of its closed subvarieties are of general type. Moreover, if $L/k$ is a field extension, then $X$ is hyperbolic over $k$ if and only if $X_L$ is hyperbolic over $L$.

For example,  a smooth proper geometrically connected  curve $X$ over $k$ is hyperbolic if and only if the genus of $X$ is at least two. 
	If $X$ is a smooth projective scheme over $\mathbb{C}$ such that the associated complex manifold  $X^{\textrm{an}}$ admits an immersive period map  (i.e., there exists a polarized variation of $\ZZ$-Hodge structures over $X^{\textrm{an}}$ whose differential is injective at all points), then $X$ is hyperbolic. This follows from  the proof of \cite[Lemma~6.3]{JL2} which uses Zuo's theorem  \cite{Zuo} (cf. \cite{Brunebarbe}).  Finally, 
	let $X$ be a smooth projective scheme over $\mathbb{C}$ and suppose that there exists a smooth proper morphism $Y\to X$ whose fibres have ample canonical bundle such that, for all $a$ in $X(\mathbb C)$, the set of $b$ in $X(\mathbb C)$ with $X_a \cong X_b$ is finite. Then $X$ is hyperbolic. This is a consequence of Viehweg's conjecture for ``compact'' base varieties \cite{patakfalvi}.

\subsection{Kodaira's criterion for bigness} We assume in this section that $k$ is of characteristic zero. Recall that for a big divisor $D$ on a projective variety, there exists a positive integer $n$ such that $nD \sim_\Q A +E$, where $A$ is ample and $E$ is effective  \cite[Lemma~2.60]{km}. We state a generalization of this statement (see Lemma \ref{lem:kod_crit}) which is presumably known; we include a proof for lack of reference.

\begin{lemma}\label{lem:bigness}
	Let $\pi \colon X\to Y$ be a quasi-finite morphism of proper Deligne-Mumford stacks over $k$. Let $D$ be a divisor on $Y$. The divisor $D$ is big on $Y$ if and only if $\pi^\ast D$ is big on $X$.
\end{lemma}
\begin{proof}
	This follows from the definition of bigness, and the fact that $\pi_\ast \pi^\ast D $ is linearly equivalent to $ m D$, where $m\geq 1$ is some integer.
\end{proof}

If $D$ is a divisor on a finite type separated Deligne-Mumford stack $\mathcal X$ over $k$ with coarse space $\mathcal X\to \mathcal X^c$, then $D$ is \textit{ample} (resp. \textit{effective}) on $\mathcal X$ if there exists a positive integer $n$ such that $nD$ is the pull-back of an ample (resp. \textit{effective}) divisor on $\mathcal X^c$. Note that, if $\mathcal X$ has an ample divisor, then $\mathcal X^c$ is a quasi-projective scheme over $k$.

\begin{lemma}\label{lem:kod_crit} Let $\calX$ be a  proper   Deligne-Mumford stack   over  $k$ with projective coarse moduli space $\calX^c$.  If $D$ is a big divisor on $\calX$, then there exists a positive integer $n$ such that $nD~\sim_\Q~\calA + \calE$, where $\calA$ is ample and $\calE$ is effective. \end{lemma}
\begin{proof}
 Let $\pi \colon \calX \to \calX^{c}$ denote the morphism from $\calX$ to its coarse moduli space $\calX^c$. It follows from \cite[Proposition~6.1]{Olsson} that  there exists a positive integer $m$ such that $mD$ is $\mathbb Q$-linearly equivalent to the pullback of a divisor $D_0$ on $\calX^c$. As $mD$ is a big divisor on $\calX$,  the divisor $D_0$ is big on $\calX^c$ (Lemma \ref{lem:bigness}).     By Kodaira's criterion for bigness, there exists a positive integer $m_2$ such that  $m_2 D_0 $ is $\mathbb{Q}$-linearly equivalent to  $ A + E$, where $A$ is an ample divisor on $\calX^c$ and $E$ is an effective divisor on $\calX^c$. Write $n = m\cdot m_2$. We now  see that $nD = m \cdot m_2\cdot D \sim_{\mathbb Q} \pi^\ast m_2 D_0 \sim_{\mathbb Q} \pi^\ast (A+E)$. Since $\mathcal A\colonequals \pi^\ast A$ is ample, and $\mathcal E \colonequals \pi^\ast E$ is effective, this concludes the proof of the lemma.  
\end{proof}

\section{Vojta's conjecture for varieties and stacks} In this section, we let $k$ be a number field.
We begin by recalling Vojta's conjecture for heights of points on schemes, using   \cite{avap} and \cite{Vojta}. Our statement of the conjecture is perhaps not the most standard, but is more natural for our setting as we will need the extension of the conjecture to algebraic stacks. 

\subsection{Discriminants of fields} Before defining the conjecture, we recall discriminants of fields following Section 2 of \cite{avap}. Given a finite extension $E / k$, define the \emph{relative logarithmic discriminant} to be:
\begin{align}
d_k(E) &= \dfrac{1}{[E:k]} \log|\disc(\calO_E)| - \log|\disc(\calO_k)| 
= \dfrac{1}{[E:k]} \deg(\Omega_{\calO_E / \calO_k}),
\end{align}
where the second equality follows from the equality of ideals $(\disc(\calO_k)) = N_{k/\Q} \det \Omega_{\calO_k / \Z}.$

 \subsection{Heights}
 In this paper we will use \emph{logarithmic} (Weil) heights.  For more details, we refer the reader to \cite{bg, hs}. 
 
 \begin{definition}\label{def:ht} Let $d $ be the degree of $k$ over $\Q$  and let $M_k$ be a complete set of normalized absolute values on $k$. The (logarithmic) height of a point $P = \left[x_0 : \dots : x_n\right]~\in~\PP^n(k)$ is defined to be: $$h_k(P) = \frac{1}{d} \sum_{v \in M_k} \log(\max_{0 \leq i \leq n} \{ \|x_i\|_v \}).$$ \end{definition}
 
 If $X$ is a projective variety with a projective embedding  $\phi: X \hookrightarrow \PP^n$, we can define a height function $h_\phi \colon X \to \mathbb{R}$ given by $$h_\phi(P) = h(\phi(P)).$$ More generally, given a very ample divisor $D$ on $X$, we define $h_D(P) = h(\phi_D(P))$, where $\phi_D$ is the natural embedding of $X$ in $\PP^n$ given by $D$. (We stress that $h_D$ is well-defined, up to a  bounded function.)
 
  \begin{prop}\label{properties1} The following statements hold.
 \begin{enumerate}
 	\item If $f \colon X \to Y$ is a morphism, then $h_{X,f^*D} = h_{Y,D} + O(1)$. 
 	\item If $D$ and $E$ are both divisors, then $h_{D+E} = h_D + h_E + O(1)$. 
 	\item If $D$ is effective, $h_D \geq O(1)$ for all points not in the base locus of $D$.
 \end{enumerate} \end{prop}
 \begin{proof}
See  \cite[Theorems~B.3.2.b,~B.3.2.c,~and~B.3.2.e]{hs}.
 \end{proof}

\subsection{Vojta's conjecture} We now state Vojta's conjecture for schemes. We stress that this conjecture (Conjecture \ref{conj:Vojta}) implies a version for stacks; see Proposition \ref{prop:vstacks}.

\begin{conj}[Vojta]\cite[Conjecture~2.3]{Vojta}\label{conj:Vojta} Let $X$ be a smooth projective scheme over $k$. Let $H$ be a big line bundle on $X$, let $r$ be a positive integer, and fix $\delta > 0$. Then there exists a proper Zariski closed subset $Z \subset X$   such that,  for all closed points $x \in  X$ with $x\not\in Z$ and $[k(x):k]\leq r$, $$ h_{K_X}(x) - \delta h_H(x)  \leq d_k(k(x))  + O(1).$$ \end{conj}

Note that the discriminant term $d_k(k(x))$ equals zero when $x$ is a $k$-rational point of $X$.

\subsection{Vojta's conjecture for stacks}
Before stating the version of Vojta's conjecture for Deligne-Mumford stacks, we introduce some preliminaries, following Section 3 of \cite{avap}. If $S$ is a finite set of finite places of $k$, we let $\calO_{k,S}$ be the ring of $S$-integers in $k$.

\subsubsection{The stacky discriminant} Let $\calX \to \Spec(\calO_{k,S})$ be a   finite type separated Deligne-Mumford stack with generic fibre $X\to \Spec k$. 
Given a point $x \in \calX(\overline{k}) = X(\overline{k})$, we define   $\calT_x \to \calX$ to be the normalization of the closure of $x$ in $\mathcal X$. Note that $\calT_x$ is a normal proper Deligne-Mumford stack over $\calO_{k,S}$ whose coarse moduli scheme is $\Spec(\calO_{k(x), S_{k(x)}})$. 
Here $S_{k(x)}$ is the set of finite places of $k(x)$ lying over $S$.

\subsubsection{Relative discriminants for stacks}\label{section:discr}

Let $E$ be a finite field extension of $k$, and let $\mathcal T$ be a normal separated Deligne-Mumford stack over $\mathcal O_E$ whose coarse moduli scheme is $\Spec  \calO_E$. We define  the relative discriminant of $\mathcal T$ over $\calO_k$ as follows: \begin{align} d_k(\calT) &= \dfrac{1}{\deg(\calT / \calO_k)}\deg(\Omega_{\calT / \Spec(\calO_k)}). \end{align}
Note that $d_k(\mathcal T)$ is a well-defined real number, and  that $\exp(d_k(\mathcal{T}))$ is  a rational number.

 \subsubsection{Heights on stacks}\label{section:heights_on_stacks}

 Let $X$ be a finite type   Deligne-Mumford stack over $k$ with finite inertia whose   coarse space $X^c$ is  a quasi-projective scheme over $k$. Fix a finite set of finite places $S$ of $k$ and a finite type separated Deligne-Mumford stack $\mathcal X\to \Spec(\calO_{k,S})$ such that $\calX_k \cong X$. Let $H$ be a divisor on $X$. Let $n\geq 1$ be an integer such that $nH$ is the pull-back of a divisor $H^c$ on $X^c$. Fix a height function $h_{H^c}$ for $H^c$ on $X^c$. We define the height function $h_H$ on $X(k)$ with respect to $H$ to be  $$h_H(x) \colonequals \frac{1}{n} h_{H^c}(\pi(x)).$$
 Note that $h_H$ is a well-defined function on $X(\overline{k})$ which is independent of the choice of $n$ and $H^c$.
 
 We now give another way to compute the height function, under suitable assumptions on $X$.
  By \cite[Theorem~2.1]{kv}, a finite type separated Deligne-Mumford stack   over $k$ which is a quotient stack and has a quasi-projective coarse moduli space admits a finite flat surjective morphism $f \colon Y \to \calX$, where $Y$ is a   quasi-projective scheme. Fix a height function $h_{f^\ast H}$ on $Y$. We define the height $h_H(x)$ of $x\in \calX(\overline{k})$ as follows.  If  $x \in \calX(\overline{k})$, then we choose   $y \in Y(\overline{k})$ to be a point over $x$, and we define  $$h_H(x) \colonequals h_{f^*(H)}(y).$$ 
It follows from the projection formula (which holds for Deligne-Mumford stacks, in particular see the introduction of \cite{Vistoli}) that $h_H$ is a well-defined function on $\mathcal X(\overline{k})$.  Moreover, if $H$ is ample, for all $d\geq 1$ and $C\in \mathbb R$, the set of isomorphism classes of $\overline{k}$-points $x$ of $\mathcal X$  such that $h_H(x) \leq C$ and $[k(x):k] \leq d$ is finite. 
 The analogous finiteness statement for $k$-isomorphism classes can fail. However, the set of $k$-isomorphism classes of $k$-points $x$ of $\mathcal X$ such that $h_H(x) + d_k(\mathcal T_x) \leq C$ and $[k(x):k] \leq d$ is finite. In particular, as $h_H(x) + d_k(\mathcal T_x)$ has the Northcott property, the expression $h_H(x) + d_k(\mathcal T_x)$ can be considered as ``the'' height of $x$  \cite{avap}.

\begin{prop}[Vojta's Conjecture for stacks]\label{prop:vstacks}  Assume Conjecture \ref{conj:Vojta} holds and fix $\delta > 0$.  Let $S$ be a finite set of finite places of $k$. Let $\calX$ be a smooth proper Deligne-Mumford stack over $\calO_{k,S}$ whose generic fibre $X = \calX_k$ is geometrically irreducible over $k$ and has a projective coarse space. Let $H$ be a big line bundle on $X$. Then,  there is a proper Zariski closed substack $Z \subset X$  such that, for all $x\in (X\setminus Z)(k)$ the following inequality holds $$h_{K_X }(x) - \delta h_H(x) \leq d_k(\calT_x)  + O(1).$$  \end{prop}
\begin{proof}
This is \cite[Proposition~3.2]{avap}.
\end{proof}

\section{Applying the stacky Vojta conjecture}

 We prove a generalization of    \cite[Proposition~2.5.1]{ih2} to morphisms of proper Deligne-Mumford stacks, under suitable assumptions. We stress that our reasoning follows Ih's  arguments in \emph{loc.~cit.} in several parts of the proof.

Let $k$ be a number field,  and let $f:\mathcal X\to \mathcal Y$ be a proper morphism of proper integral Deligne-Mumford stacks over $B = \Spec \calO_{k,S}$, where $\mathcal{X}$ is  smooth   with a projective coarse moduli space.  Let $h$ be a height function on $\calX$ and let $h_\calY$ be a height function on $\calY$ associated to an ample divisor such that $h_\calY \geq 1.$ Let $\eta$ be the generic point of $\mathcal{Y}$, let $\calX_{\eta}$ be the generic fibre of $f \colon \calX\to \calY$, and let $\calX_k$ be the generic fibre of $\calX \to B$.  Note that $\mathcal X_k$ is a smooth proper Deligne-Mumford stack over $k$ with a projective coarse space.

\begin{prop}	 \label{prop:mainprop} 
	 Assume Conjecture  \ref{conj:Vojta}.  Suppose that the morphism $f$ is representable by schemes, and that $\calX_\eta$ is smooth and of general type. Then
	 there exists a real number $c(k,S, \calY,f)$ and a proper Zariski closed substack $\calZ \subset \calX$ such that, for all $P$ in $\calX (B)\setminus \calZ$, the following inequality holds:  $$h(P) \leq  c(k, \calY,f) \cdot \big(d_k(\mathcal T_P)+ h_\calY(f(P))\big).$$
 \end{prop}

\begin{proof} 	
	Let $\Delta$ be an ample divisor on $\calX$ such that the associated height $h_\Delta$ on $\calX$ satisfies $h_\Delta \geq 1$. Note that the push-forward of $\Delta$ to the coarse space is ample.  Recall that $\mathcal{X}_k$ denotes the generic fibre of $\mathcal{X}\to B$. 
	Moreover, Vojta's conjecture (Conjecture \ref{conj:Vojta}) implies Vojta's conjecture for stacks (Proposition \ref{prop:vstacks}).  Therefore, by Vojta's conjecture for stacks (Proposition \ref{prop:vstacks}) applied to $\calX_k$, 
	there exists a proper Zariski closed substack $Z\subset \calX_k$ such that, for all $P \in \calX_k(k)\setminus Z$, the following inequality \[	h_{K_{\calX_k}}(P) - \dfrac{1}{2} \epsilon h_{\Delta}(P) \leq d_k(\calT_P) + O(1) \] holds, where we compute all invariants with respect to the model $\mathcal X$ for $\calX_k$ over $B$.
	In particular, there exists a proper closed substack  $\calZ$ of $\calX$ (namely, the closure of $Z$ in $\mathcal X$) such that, for all $P$ in $\calX(B) $ not in $\calZ$, the following inequality holds
	\begin{eqnarray}\label{eqn:vojta}
	h_{K_\calX}(P) - \dfrac{1}{2} \epsilon h_{\Delta}(P) \leq d_k(\calT_P) + O(1).
	\end{eqnarray}

Since $f$ is representable, $\calX_\eta$ is a scheme. Moreover, since $\calX_\eta$ is smooth and of general type, by the Kodaira criterion for bigness (Lemma \ref{lem:kod_crit}), there exists an ample divisor $A$ on $\calX_\eta$, an effective divisor $E$ on $\calX_\eta$, and a positive integer $n$ such that $$n(K_{\calX_\eta}) \sim_\Q  A + E.$$ For a small enough $\epsilon \in \Q _{>0}$, we can rewrite 
	\begin{align*} (K_\calX - \epsilon\Delta)|_\eta &= K_{\calX_\eta}  - \epsilon\Delta|_\eta
	 \sim_{\mathbb Q} \bigg(\dfrac{1}{n}A + \dfrac{1}{n}E\bigg) - \epsilon\Delta|_\eta \\
	 &= \bigg(\dfrac{1}{n}A - \epsilon\Delta|_\eta\bigg) + \dfrac{1}{n}E.
	\end{align*}  
	Thus, there exists an effective divisor $E'$ on $\calX_\eta$ and a positive integer $m$ such that  $$m\bigg(\bigg(\dfrac{1}{n}A -\epsilon\Delta|_\eta\bigg) + \dfrac{1}{n}E\bigg) \sim_{\mathbb Q} E'. 
 $$
  Taking Zariski closures of these divisors in $\calX$, it follows that there exists a vertical $\Q$-divisor $\calF$ on $\calX$ and an effective divisor $\mathcal E$ on $\calX$  such that $$K_\calX  -\epsilon \Delta + \calF \sim_{\mathbb Q} \dfrac{1}{m}\calE. $$

	 Since $\mathcal F$ is a vertical divisor on $\calX$,  there is an effective divisor $\mathcal G$ on $\calY$ such that $\mathcal F \leq f^*\mathcal G$. Therefore, by Proposition \ref{properties1},   the inequality $h_{\mathcal F} \leq  h_{f^*\mathcal{G}} +O(1)$ holds, outside of $\mathrm{Supp} \ (f^*\mathcal{G})$, and $h_{f^*\mathcal{G}} = (h_{\mathcal G} \circ f) + O(1)$. In particular, since $h_\calY$ is a height associated to an ample divisor,  we see that $h_{\mathcal{G}} \leq   O(h_\calY)$  by 	\cite[Proposition~5.4]{lang}. Therefore, for all  points $t$ in $\calY(k)$  and all $P \in \calX_t(B) \setminus \textrm{Supp}(f^*G)$, the inequality 
	 \begin{eqnarray*}
	 h_{\mathcal {F}}(P) & \leq  & h_{f^*\mathcal {G}}(P) + O(1)= h_{\mathcal G} ( f(P)) + O(1)\leq O\big(h_\calY(f(P))\big) + O(1)
	 \end{eqnarray*} 
	 holds, outside  of  $\mathrm{Supp} \ (f^*\mathcal{G})$. In particular, replacing $\mathcal Z$ by the  union of $\mathcal Z$ with $\mathrm{Supp}(f^\ast \mathcal G)$,   it follows that 
	  \begin{eqnarray}\label{eqn:vertical}h_{\mathcal {F}} \leq O(h_\calY \circ f) + O(1) 
	\end{eqnarray} outside $\mathcal Z$.
	Since $K_{\mathcal X} - \epsilon \Delta +F$ is effective, it follows that, replacing $\mathcal Z$ by a larger proper closed substack of $\mathcal X$ if necessary,   the inequality 
	\begin{eqnarray}\label{eqn:effectivity}
	h_{K_\calX - \epsilon \Delta + F} \geq O(1) 
	\end{eqnarray}
	holds outside $\mathcal Z$ by Proposition \ref{properties1} (3).
	
	Let $d_k(\mathcal T)$ be the function that assigns to a point $P$ in $\mathcal X(\overline{k})$ the real number $d_k(\mathcal T_P)$.  In particular, we obtain that
	\begin{align*} O(1) &\leq h_{K_\calX  -\epsilon \Delta +  F}\leq (h_{K_{\calX}} - \frac{1}{2} \epsilon h_\Delta) - \frac{1}{2} \epsilon h_\Delta + h_F + O(1)\\ &  \leq   (h_{K_{\calX}}   - \frac{1}{2} \epsilon h_\Delta) - \frac{1}{2} \epsilon h_\Delta + O(h_\calY\circ f)  + O(1) 
	 \leq  d_{k}(\mathcal T)  - \frac{1}{2} \epsilon h_\Delta + O(h_\calY \circ f)  + O(1),
	\end{align*}
	where the inequalities follow from Equation (\ref{eqn:effectivity}),   Proposition \ref{properties1}.(2),  Equation (\ref{eqn:vertical}), and Vojta's conjecture (\ref{eqn:vojta}) respectively. 
	
	We conclude  that, for all $t$ in $\mathcal Y(B)$ and all $P$ in $\calX_t(B) \setminus \calZ$ the inequality
	
	 $$\frac{1}{2} \epsilon h_\Delta(P) \leq d_k(\mathcal T_P)+ O(h_\calY(t)) + O(1)$$ holds.
	 Therefore, there is a real number $c>0$ such that, for all $t$ in $\mathcal Y(t)$ and all $P$ in $\mathcal X_t$ not in $\mathcal Z$,  the inequality $$h_\Delta(P) \leq c\cdot \bigg(d_k(\mathcal T_P)+ O\big(h_\calY(t)\big)\bigg) + O(1)$$ holds. In particular, replacing $c$ by a larger real number if necessary, we conclude that $$ h_\Delta(P) \leq  c\cdot \bigg(d_k(\mathcal T_P)+ h_\calY(t)\bigg) + O(1).$$    As $\Delta$ is ample and $h_\Delta \geq 1$, we conclude that, using 	\cite[Proposition~5.4]{lang} and replacing $c$ by a larger real number if necessary, for all $t$ in $\mathcal Y(t)$ and all $P$ in $\mathcal X_t$ not in $\mathcal Z$, the inequality
	  $$h(P) \leq O\big(h_\Delta(P)\big) \leq c \cdot \big(d_k(\mathcal T_P) + h_{\calY}(f(P))\big) + O(1)$$  holds. In particular, replacing $c$ by a larger real number $c(k,\mathcal Y, f)$ if necessary, we conclude that the following inequality   $$h(P) \leq  c(k,\calY,f) \cdot (d_k(\mathcal T_P)+ h_\calY(f(P)))$$ holds. 
	  \end{proof}

\section{Uniformity results}
Let $k$ be  a number field.
In this section we prove Theorem \ref{thm:stackthm}.

\begin{lemma}\label{lem:desing} Let $f \colon X \to Y$ be a proper surjective morphism of  proper   Deligne-Mumford stacks over $k$ which is representable by schemes.  
	Let $h$ be a height function on $X$ and let $h_Y$ be a height function on $Y$ associated to an ample divisor with $h_Y \geq 1$. 
	Assume Conjecture \ref{conj:Vojta}. Suppose that the generic fibre $X_\eta$ of $f\colon X\to Y$ is smooth and of general type. There exists a proper Zariski closed substack $Z\subset X$ and a real number $c$ depending only on $k$, $X$, $Y$, and $f$, such that, for all $P$ in $X(k)\setminus Z$, the following inequality holds
	\[
	h(P) \leq c \cdot \big(h_Y(f(P)) + d_k(\calT_P)\big).
	\] 
\end{lemma}
\begin{proof} We may and do assume that $X$ and $Y$ are geometrically integral over $k$.

	Let $\mu \colon \wtilde{X}\to X$ be a desingularization of $X$; see \cite[Theorem~5.3.2]{temkin}. Note that $\wtilde{f} \colon \wtilde{X} \to Y$ is a proper surjective morphism of proper Deligne-Mumford stacks whose generic fibre is of general type. Define $X_{exc}\subset X$ to be the exceptional locus of $\mu \colon \wtilde{X}\to X$, so that $\mu$ induces an isomorphism of stacks from $\wtilde{X}\setminus \mu^{-1}(X_{exc})$ to $X\setminus X_{exc}$.
	Note that $X_{exc}$ is a proper closed substack of $X$, as $X$ is reduced.
	
	Let $\wtilde{h}$ be the height function on $\wtilde{X}$ associated to $h$, so that, for all $\wtilde{P}$ in $\wtilde{X}$, we have $\wtilde{h}(\wtilde{P}) = h(P)$. As we are assuming Conjecture \ref{conj:Vojta}, it follows from  Proposition \ref{prop:mainprop} that there exists a proper Zariski closed substack $\wtilde{Z}\subset \wtilde{X}$ such that, for all  $\wtilde{P}$ in $\wtilde{X}(k)\setminus \wtilde{Z}$, the following inequality 
	\[
	\wtilde{h}(P) \leq c \cdot \big( h_Y(\wtilde{f}(P)) + d_k(\mathcal T_P)\big)
	\]
	holds, where $c$ is a real number depending only on $k$, $Y$, $X$, and $f$. (Here we use that $\wtilde{X}\to X$ only depends on $X$.)

	Define $Z$ to be the closed substack $\mu(\wtilde{Z})\cup X_{exc}$ in $X$.  Note that $\mu$ induces an isomorphism from $\wtilde{X}\setminus \mu^{-1}(Z) $ to $X\setminus Z$. Therefore, we conclude that, for all $P$ in $X(k)\setminus Z$, the inequality 
	\[ h(P) = \wtilde{h}(\wtilde{P}) \leq c \cdot \big(h_Y(f(P)) +  d_k(\mathcal T_P)\big) \] holds, where $\wtilde{P}$ is the unique point in $\wtilde{X}$ mapping to $X$.
\end{proof}

\begin{proof}[Proof of Theorem \ref{thm:stackthm}] 
	Since $U$ is constructible, we have that $U = \cup_{i=1}^n U_i$ is a finite union of locally closed substacks $U_i \subset Y$. Let $Y_i$ be the closure of $U_i$ in $Y$,   let $X_i = X\times_Y Y_i$, and let $f_i: X_i \to Y_i$ be the associated morphism. Note that $U_i$ is open in $Y_i$. In particular, to prove the theorem, replacing $X$ by $X_i$, $Y$ by $Y_i$,   $U$ by $U_i$, and $f:X\to Y$ by $f_i: X_i\to Y_i$ if necessary, we may and do assume that $U$ is open in $Y$.
	
	We now argue by induction on $\dim X$. If $\dim X = 0$, then the statement is clear.
	
As we are assuming Conjecture \ref{conj:Vojta}, it follows from  Lemma \ref{lem:desing} that there exists a proper Zariski closed substack $Z \subset X$ and a real number $c_0>0$ depending only on $k$, $X$, $Y$, and $f$ such that, for all $P$ in $X(k)\setminus Z$, the inequality 
\begin{eqnarray}\label{ineq1}
h(P) & \leq &  c_0\cdot \big( h_Y(f(P)) + d_k(\mathcal T_P)\big) 
\end{eqnarray}
holds.

Let $X_1,\ldots, X_s\subset Z$ be the irreducible components of $Z$. For $i\in \{1,\ldots,s\}$, let $Y_i = f(X_i)$ be the image of $Y_i$ in $Y$. Note that $f_i \colonequals f|_{X_i} \colon X_i\to Y_i$ is a proper morphism of proper integral Deligne-Mumford stacks which is representable by schemes. Moreover, for  $t$ in the open subscheme $Y_i\cap U$ of $Y_i$, the proper variety $X_{i,t}$ is hyperbolic, as $X_{i,t}$ is a closed subvariety of the hyperbolic variety $X_t$. Let $h_i$ be the restriction of $h$ to $X_i$, and let $h_{Y_i}$ be the restriction of $h_Y$ to $Y_i$. 

Since $X_i$ is a proper Zariski closed substack of $X$, it  follows that $\dim X_i < \dim X$. Therefore, by the induction hypothesis, we conclude that there is a real number $c_i>0$ depending only on $k$, $X_i$, $Y_i$, and $f_i$  such that, for all $P$ in $X_i(k)$, 
the following inequality 
\begin{eqnarray}\label{ineq2} 
h(P) = h_i(P) & \leq & c_i \cdot \big(h_{Y_i}(f_i(P)) + d_k(\mathcal T_{P})\big) = c_i \cdot \big(h_Y(f(P)) + d_k(\mathcal T_P)\big).
\end{eqnarray} holds. Let $c' \colonequals \max(c_1,\ldots,c_s)$. By (\ref{ineq2}),  we conclude that, for all $P$ in $Z(k)$, the inequality
\begin{eqnarray}\label{ineq3}
h(P) &\leq & c'\cdot \big( h_Y(f(P)) + d_k(\mathcal T_P)\big)
\end{eqnarray} holds. 

Combining (\ref{ineq1}) and (\ref{ineq3}), we conclude the proof of the theorem with $c\colonequals \max(c_0,c')$.
\end{proof}

\begin{lemma}\label{lem:disc_for_rep}
 Let  $f \colon \calX \to \calY$ be a proper surjective morphism of  proper   Deligne-Mumford stacks over $\mathcal{O}_K$ which is representable by schemes.  If $P\in \calX(k)$, then  $d_k(\mathcal{T}_P) = d_k(\mathcal{T}_{f(P))})$.
\end{lemma}
\begin{proof} Since the normalization morphism of an integral algebraic stack is representable and $\calX\to\calY$ is representable, we see that the morphism $\mathcal{T}_P\to\mathcal{T}_{f(P)}$ is representable. Therefore, we see that $\mathcal{T}_P\to \mathcal{T}_{f(P)}$ is proper    surjective and representable by schemes. Moreover, it is birational, and it  has a section generically. This implies that 
the morphism $\mathcal{T}_P\to \mathcal{T}_{f(P)}$ is a proper birational quasi-finite representable morphism. Since $\mathcal{T}_{f(P)}$ is normal, the lemma follows from Zariski's Main Theorem for stacks. 
\end{proof}
  
\begin{cor}\label{cor:thm}
Let $f \colon \calX \to \calY$ be a proper surjective morphism of  proper   Deligne-Mumford stacks over $\mathbb{Z}$ which is representable by schemes. Let $X:=\calX_{\mathbb{Q}}$ and $Y:=\calY_{\mathbb{Q}}$. Let $h$ be a height function on $X$ and let $h_Y$ be a height function on $Y$ associated to an ample divisor with $h_Y \geq 1$. 
Assume Vojta's height conjecture (Conjecture \ref{conj:Vojta}). Let $U\subset Y$ be a constructible substack such that, for all $t\in U$, the variety $X_t$ is smooth and hyperbolic.  Then there is a real number $c > 0$ depending only on $k$, $Y$, $X$, and $f$ such that, for all   $P$ in $X(k)$ with $f(P)$ in  $U$, the following inequality holds \[ h(P) \leq c  \cdot \big(h_Y(f(P)) + d_k(\calT_{f(P)})\big). \] 
\end{cor}
\begin{proof}
Combine  Theorem \ref{thm:stackthm} and Lemma \ref{lem:disc_for_rep}.
\end{proof}

\section{Applications}\label{section:apps}
In this section we apply our main result (Theorem \ref{thm:stackthm}) to some explicit families of hyperbolic varieties, and  prove Theorems \ref{thm1} and \ref{thm2}.

\subsection{Application to curves}\label{section:apps_curves}

For $g\geq 2$ an integer, let $\mathcal M_g$ be the stack over $\mathbb Z$ of smooth proper genus $g$ curves. Let $\overline{\mathcal M}_g$ be the stack of stable genus $g$ curves. Note that $\mathcal M_g$ and $\overline{\mathcal M}_g$ are smooth finite type separated Deligne-Mumford stacks. Moreover, $\mathcal M_g\to \overline{\mathcal M}_g$ is an open immersion, and $\overline{\mathcal M}_g$ is proper over $\mathbb Z$ with a projective coarse space \cite[Theorem~5.1]{Kollar}. These properties of $\mathcal M_g$ and $\overline{\mathcal M}_g$ are proven in \cite{DeligneMumford}. We fix an ample divisor $H$ on $\overline{\mathcal M}_g$.

If $X$ is a smooth projective curve of genus at least two over a number field $k$, we let $h \colon X(\overline{k})\to \mathbb R$ be the height with respect to the canonical embedding $X\to \mathbb P^{5g-6}_k$. Moreover, we define the height of  $X$ to be the height of the corresponding $k$-rational point of $\overline{\mathcal M}_g$ with respect to the fixed ample divisor $H$ on $\overline{\mathcal M}_g$ (following Section \ref{section:heights_on_stacks}). 

If $X$ is a smooth projective curve of genus at least two over $k$ and $P$ is a $k$-rational point of $X$, then the pair $(X,P)$ defines a point on the universal curve $\mathcal{M}_{g,1}$ over $\mathcal{M}_g$.  We let $d_k(\mathcal T_{(X,P)})$ denote the discriminant of this point, as defined in Section \ref{section:discr}.

\begin{proof}[Proof of Theorem \ref{thm1}] Since $U:=\mathcal M_g$ is open in $Y:=\overline{\mathcal M_g}$, we can apply Corolary \ref{cor:thm} to the universal family of stable genus $g$ curves $f \colon X \to Y$.
\end{proof}

\begin{remark}\label{rem:faltings_height}
In Theorem \ref{thm1}, one can also use the (stable) Faltings height $h_{\mathrm{Fal}}(X)$ of $X$ (instead of the height $h$ introduced above). Indeed, it follows from \cite{Faltings, Pazuki} that the Faltings height $h_{\mathrm{Fal}}(X)$ is bounded by $ h(X) +c$, where $c$ is a real number depending only on the genus of $X$.
\end{remark}

\subsection{Hyperbolic surfaces}
Recall that, if  $S$ is a smooth projective surface, then $c_1^2(S) = K_S^2$ and $c_2(S) = e(S)$ is the topological Euler characteristic. Moreover, by Noether's lemma, they are related by the following equality: $$ \chi(S,\mathcal O_S) = \dfrac{c_1(S)^2 + c_2(S)^2}{12}.$$ In particular, the information of $K_S^2$ and $\chi(S)$ is equivalent to the data of $c_1(S)$ and $c_2(S)$. Finally, we note that $c_2(S) \geq  1$ for any surface of general type $S$  \cite[X.1~and~X.4]{beauville}. 

A smooth proper morphism $f:X\to Y$ of schemes is a canonically polarized smooth surface over $Y$ if, for all $y$ in $Y$, the scheme $X_y$ is  a connected two-dimensional scheme and $\omega_{X_y/k(y)}$ is ample.
If $a$ and $b$ are integers, we let $\mathcal M_{a,b}$ over $\mathbb Z$ be the stack of smooth canonically polarized surfaces
  $S$ with $c_1(S)^2 = a $ and $c_2(S)  = b$.  Note that $\mathcal M_{a,b}$ is a finite type algebraic stack over $\mathbb Z$ with  finite diagonal (cf. \cite{mm, tankeev}).

 \begin{lemma}\label{lem:smoothness_argument}   
 	If $S $ is   a smooth   hyperbolic surface over a field $k$, then $S$ is canonically polarized.  
 \end{lemma}
 \begin{proof}
If $S$ is a (smooth) minimal surface of general type, then the canonical model $S^{c}$ is obtained by contracting all rational curves with self intersection $-2$ \cite[Chapter~9]{Liu}. Consequently, the singularities on a singular surface in $\mathcal M_{a,b}(k)$ are rational double points arising from the contraction of these $-2$ curves. As having a $-2$ rational curve would contradict $S$ being hyperbolic, we see that $S^{c}$ must be smooth, and thus equal to $S$. As the canonical bundle on $S^c$ is ample, we conclude that $S$ is canonically polarized.  
 \end{proof}

  Let $\mathcal M_{a,b}^{h}\subset \mathcal M_{a,b}$ be the substack of hyperbolic surfaces, i.e., for a scheme $S$, the objects $f:X\to S$ of the full subcategory $\mathcal M_{a,b}^{h}(S)$ of $\mathcal M_{a,b}(S)$ satisfy the property that, for all $s$ in $S$, the surface $X_s$ is hyperbolic  (Definition \ref{defn:hyperbolic}).  
We do not know of any result on the algebraicity of $\mathcal M_{a,b}^h$ (nor the algebraicity of $\mathcal M_{a,b}^h\times_\ZZ\Spec\CC$). However, if $S$ is a minimal projective surface of general type  over $\mathbb C$ and $c_1^2(S) > c_2(S)$, then Bogomolov proved \cite{bogomolov} that $S$ contains only a finite number of curves of bounded genus, and thus $S$ contains only finitely many rational and elliptic curves.    In  \cite[Theorem~1.1]{miyaoka} Miyaoka proved a more effective version of Bogomolov's result, showing that in fact the canonical degree of such curves is bounded in terms of $c_1^2$ and $c_2$. Using these results we are able to prove the following.

\begin{lemma}\label{lem:constructibility}
	If $a>b$, then $\mathcal M_{a,b}^h\times_{\ZZ}\Spec \mathbb{C}$ is a constructible substack of $\calM_{a,b} \times_{\ZZ} \Spec \mathbb{C}$.
\end{lemma}
\begin{proof} Let $a$ and $b$ be integers such that $a>b$.
	Let $N$ be an integer such that, for all $S$ in $\mathcal M_{a,b}(\mathbb C)$, the ample line bundle $\omega_{S/\mathbb C}^{\otimes N}$ is very ample. In particular,  $S$ is embedded in $\mathbb P^n \cong \mathbb P(\mathrm{H}^0(S,\omega_{S/\mathbb C}^{\otimes N}))$. Let $\mathrm{Hilb}_{a,b}$ be the   Hilbert scheme of $N$-canonically embedded smooth surfaces, and note that $\mathcal M_{a,b} = [\mathrm{Hilb}_{a,b}/\mathrm{PGL}_{n+1}]$. 
	
	Let $\mathrm{H}_d$ be the Hilbert scheme of    (possibly singular) curves of canonical degree $d$ in $\mathbb P^n$.  Let $\mathrm{H}_d^{\textrm{int}}$ be the subfunctor of geometrically integral curves. Since the universal family over $\mathrm{H}_d$ is flat and proper, the subfunctor  $\mathrm{H}_d^{\textrm{int}}$ is an open subscheme   of $\mathrm{H}_d$; see \cite[Appendix~E.1.(12)]{GW}.


 Let $\calW_{a,b,d} \subset \mathrm{H}_d^{\textrm{int}}\times \mathrm{Hilb}_{a,b}$ be the incidence correspondence subscheme parametrizing  parametrizing pairs $(C,S)$ where the curve $C$ is inside the surface $S$. (Note that $\calW_{a,b,d}$ is a closed subscheme of $\mathrm{H}_d^{\textrm{int}}\times \mathrm{Hilb}_{a,b}$.)  
 
 By Miyaoka's theorem \cite[Theorem~1.1]{miyaoka}, there exist integers $d_1,\ldots,d_m$ which depend only on $a$ and $b$ with the following property. A surface $S \in \mathcal M_{a,b}(\mathbb{C})$ is hyperbolic if and only if, for all $i=1,\ldots,m$, it does not contain an integral curve of degree $d_i$.

 Note that, by Chevalley's theorem, for all $d\in \ZZ$, the image of the composed morphism $$\calW_{a,b,d}  \subset \mathrm{H}_d \times\mathrm{Hilb}_{a,b}\to \mathrm{Hilb}_{a,b}\to \mathcal M_{a,b}$$ is constructible. Let $\mathcal M_{a,b,d_i}$ be the stack-theoretic image of $\calW_{a,b,d_i}$ in $\mathcal M_{a,b}$. Since a finite union of constructible substacks is constructible, the union 
 $\bigcup_{i=1}^m \mathcal M_{a,b,d_i} $  is a constructible substack of $\mathcal M_{a,b}$.
 
 Finally, by construction,  a surface $S$ in $\mathcal M_{a,b}(\mathbb C)$ is hyperbolic if and only if it is not (isomorphic to an object) in the constructible substack  $\bigcup_{i=1}^m \mathcal M_{a,b,d_i} $.  As the complement of a constructible substack is constructible, we conclude that $\mathcal M_{a,b}\times_\ZZ \Spec \mathbb{C}$ is a constructible substack of $\mathcal M_{a,b}\times_\ZZ\mathbb{C}$.
 \end{proof}

Let $\mathcal{U}_{a,b}\to \mathcal{M}_{a,b}$ denote the universal family.
 We let  $\overline{\mathcal M_{a,b,\mathbb Q}}$ be a compactification of $\mathcal M_{a,b,\mathbb Q}$ with a projective coarse moduli space;  see   \cite[Section 2.5]{hacking} (or \cite[Corollary~5.6]{Kollar}) for an explicit construction of such a compactification. (As the stack of smooth canonically polarized surfaces is open in the stack of canonical models,  it suffices to compactify the latter, as is achieved in \emph{loc. cit.} for all $a$ and $b$.)   We now choose $\overline{\mathcal M}_{a,b}$ to be a compactification of $\mathcal M_{a,b}$ over $\ZZ$ whose generic fibre $\overline{\mathcal M}_{a,b}\times_\ZZ\Spec \QQ$ is isomorphic to  $\overline{\mathcal M_{a,b,\QQ}}$, and we also choose a representable proper morphism $\overline{\mathcal{U}}_{a,b}\to \overline{\mathcal{M}}_{a,b}$ extending the universal family over $\mathcal{M}_{a,b}$ compatibly with the universal family over $\overline{\mathcal{M}_{a,b,\mathbb{Q}}}$.

If $S$ is a smooth projective canonically polarized hyperbolic surface  over a number field $k$, we let $h:S(\overline{k})\to \mathbb R$ be the height with respect to the very ample divisor $\omega_{S/k}^{\otimes 34}$ (see \cite{tankeev}). Moreover, we define the height of  $S$ in $\mathcal M_{a,b,\QQ}(k)$ to be the height of the corresponding $k$-rational point of $\overline{\mathcal M}_{a,b}$ with respect to some fixed ample divisor $H$ on $\overline{\mathcal M}_{a,b,\QQ}$ (following Section \ref{section:heights_on_stacks}). 


If $S$ is a smooth projective surface   and $P$ is a $k$-rational point of $S$, then the pair $(S,P)$ defines a point on the universal surface $\mathcal{U}_{a,b}$ over $\mathcal{M}_{a,b}$.  We let $d_k(\mathcal T_{(S,P)})$ denote the discriminant of this point, as defined in Section \ref{section:discr}.

\begin{proof}[Proof of Theorem \ref{thm2}] 
	By Lemma \ref{lem:constructibility} and standard  descent arguments (cf. \cite[Tag~02YJ]{stacks-project}), we conclude that  $\calM^h_{a,b} \times_{\ZZ} \Spec \mathbb{Q}$ is a constructible substack of $\calM_{a,b}\times_{\ZZ}\Spec \mathbb{Q}$. Also, a smooth hyperbolic surface is canonically polarized by Lemma \ref{lem:smoothness_argument}. Therefore, the result follows from an application of Corollary \ref{cor:thm} to the universal family over $\calY := \overline{\mathcal{M}_{a,b}}$, with  $Y:= \overline{\mathcal M}_{a,b,\QQ}$, and the constructible substack $U:= \mathcal M_{a,b,\mathbb{Q}}^h$ in $Y$. 
	 \end{proof}

 \begin{remark}
  There are many examples of surfaces of general type with $c_1^2 > c_2$.  Some of the simplest examples are surfaces $S$  with ample canonical bundle  such that there exist a smooth proper curve $C$ and a smooth proper morphism $S\to C$ (see for instance \cite{Kodaira}).
 \end{remark}

\bibliographystyle{alpha}
\bibliography{heights_stacks}

\end{document}